\documentclass[12pt]{article}
\usepackage[dvips]{color}
\usepackage{graphicx}
\usepackage{epsfig}
\usepackage{verbatim}
\usepackage{amsthm}
\usepackage{amssymb}
\usepackage{amsxtra}     
\usepackage{multirow}
\usepackage{hyperref}

\newtheorem{theorem}{Theorem}[section]
\newtheorem{definition}[theorem]{Definition}

\newtheorem{corollary}[theorem]{Corollary}

\DeclareSymbolFont{AMSb}{U}{msb}{m}{n}
\DeclareMathSymbol{\N}{\mathbin}{AMSb}{"4E}
\DeclareMathSymbol{\Z}{\mathbin}{AMSb}{"5A}
\DeclareMathSymbol{\R}{\mathbin}{AMSb}{"52}
\DeclareMathSymbol{\Q}{\mathbin}{AMSb}{"51}
\DeclareMathSymbol{\I}{\mathbin}{AMSb}{"49}
\DeclareMathSymbol{\C}{\mathbin}{AMSb}{"43}

\DeclareMathOperator{\sgn}{sgn}
\DeclareMathOperator{\im}{im}
\DeclareMathOperator{\link}{link}
\DeclareMathOperator{\des}{des}

\begin{document}

\title{The Coloring Complex and Cyclic Coloring Complex of a Complete $k$-Uniform Hypergraph}
\author{Sarah Crown Rundell\\Denison University} \maketitle

\begin{abstract} In this paper, we study the homology of the coloring complex and the cyclic coloring complex of a complete $k$-uniform hypergraph.  We show that the coloring complex of a complete $k$-uniform hypergraph is shellable, and we determine the rank of its unique nontrivial homology group in terms of its chromatic polynomial.  We also show that the dimension of the $(n-k-1)^{st}$ homology group of the cyclic coloring complex of a complete $k$-uniform hypergraph is given by a binomial coefficient.   Further, we discuss a complex whose $r$-faces consist of all ordered set partitions $[B_1, \hdots , B_{r+2}]$ where none of the $B_i$ contain a hyperedge of the complete $k$-uniform hypergraph $H$ and where $1 \in B_1$.  It is shown that the dimensions of the homology groups of this complex are given by binomial coefficients.  As a consequence, this result gives the dimensions of the multilinear parts of the cyclic homology groups of $\C[x_1, \hdots ,x_n]/ \{x_{i_1} \hdots x_{i_k} \mid i_{1} \hdots i_{k}$ is a hyperedge of $H \}$.
\end{abstract}

\section{Introduction}

\hspace{.5in}In this paper, we will study the homology of the coloring complex and the cyclic coloring complex of a complete $k$-uniform hypergraph.  Throughout the paper, let $G$ be a simple graph on $n$ vertices.  

\hspace{0.5in}Consider $R = A/I$ where $A = F[x_S \mid S \subseteq [n] ]$, $I$ is the ideal generated by $\{x_U x_T \mid U \not \subseteq T, T \not \subseteq U \}$, and $F$ is a field of characteristic zero.  The ideal $K_G$ is defined to be the ideal generated by the monomials $x_{X_1}^{e_1} x_{X_2}^{e_2} \hdots x_{X_l}^{e_l}, e_i > 0$ such that for all $i$, $1 \leq i \leq l+1, Y_i = X_i \backslash X_{i-1}$ does not contain an edge of $G$ ($X_0 = \emptyset$ and $X_{l+1} = [n]$).  Steingr\'{\i}msson~\cite{sm} shows that there is a bijection between the monomials of $K_G$ of degree $r$ and colorings of $G$ with $r+1$ colors.  He thus names $K_G$ the coloring ideal and notes that the quotient $R/K_G$ is the face ring of a simplicial complex, $\Lambda(G)$, or the coloring complex of $G$. 

\hspace{0.5in}Jonsson~\cite{jo} extended the work of Steingr\'{\i}msson and proved that for the case where $G$ has at least one edge, $\Lambda(G)$ is a constructible complex.  He shows that the $(n-3)^{rd}$ homology group of $\Lambda(G)$ is the only nonzero homology group and that the dimension of this group is equal to $\chi_G(-1)-1$, where $\chi_{G}(\lambda)$ is the chromatic polynomial of $G$.  

\hspace{0.5in}Crown~\cite{cr} defined and studied the cyclic coloring complex of a graph, denoted $\Delta(G)$.   In her paper, she determines the dimensions of the homology groups of $\Delta(G)$.  She shows that for a connected graph, $G$, the dimension of the $(n-3)^{rd}$ homology group of $\Delta(G)$ is equal to $n-2$ plus $\chi_{G}'(0)$, and the dimension of the $r^{th}$ homology group, for $r < n-3$, is given by the binomial coefficient $\binom{n-1}{r+1}$.  If $G$ has $k$ connected components, she shows that the dimension of the $(n-3)^{rd}$ homology group of $\Delta(G)$ is equal to $n-(k+1)$ plus $\frac{1}{k!}\mid \chi_{G}^{k}(0) \mid$, where $\chi_{G}^{k}(\lambda)$ is the $k^{th}$ derivative of $\chi_{G}(\lambda)$.  For $r < n-3$, she gives a formula for the dimension of the $r^{th}$ homology group of $\Delta(G)$ in terms of a sum of binomial coefficients and the value $\frac{1}{k!} \mid \chi_{G}^{k}(0) \mid$.

\hspace{.5in}Let $H$ be a hypergraph on $n$ vertices.  The coloring complex of a hypergraph, $\Lambda(H)$ was introduced in Long and Rundell~\cite{lr}, as well as in Breuer, Dall, and Kubitzke~\cite{bd}.  In Long and Rundell~\cite{lr}, the authors extend a result of Hanlon~\cite{ha} in which he shows that there exists a Hodge decomposition of the unique nontrivial homology group of $\Lambda(G)$ and that the dimension of the $j^{th}$ Hodge piece of this decomposition equals the absolute value of the coefficient of $\lambda^{j}$ in $\chi_{G}(\lambda)$.  Long and Rundell~\cite{lr} extend this result by showing that the Euler Characteristic of the $j^{th}$ Hodge subcomplex of $\Lambda(H)$ is related to the coefficient of $\lambda^{j}$ in $\chi_{H}(\lambda)$.  They also show that for a class of hypergraphs, which they call star hypergraphs, the coloring complex of the hypergraph is Cohen-Macaulay.  In the Breuer, Dall, and Kubitzke~\cite{bd} paper, the authors show that the $f$- and $h$- vectors of the coloring complexes of hypergraphs provide tighter bounds on the coefficients of chromatic polynomials of hypergraphs.  They also show that the coloring complex of a hypergraph has a wedge decomposition, and they provide a characterization of hypergraphs having a connected coloring complex.

\hspace{.5in}In this paper, we will study the homology of the coloring complex and the cyclic coloring complex of a complete $k$-uniform hypergraph, and we list the main results of the paper below.  First we will show:\\

\textbf{Theorem ~\ref{ShellableThm1}} \emph{Let $H$ be the complete $k$-uniform hypergraph on $n$ vertices and let $k > n/2$.  Then $\Lambda(H)$ is shellable of dimension $n-k-1$.}\\

From this theorem, we obtain a basis for the cohomology of $\tilde{H}^{n-k-1}(\Lambda(H), \Z)$ (Corollary ~\ref{Cohomology_Basis}), and we also obtain the following result:\\

\textbf{Corollary ~\ref{Homology_Complete_Hypergraph}} \emph{If $H$ is the complete, $k$-uniform hypergraph on $n$ vertices, then the homology of $\Lambda(H)$ is nonzero only in dimension $n-k-1$, and the dimension of $H_{n-k-1}(\Lambda(H))$ equals the sum of the absolute values of the coefficients of $\chi_{H}(\lambda)$ minus one.  Moreover, the dimension of the $j^{th}$ Hodge piece in the Hodge decomposition of $\Lambda(H)$ is 
$$\dim(H_{n-k-1}^{(j)}(\Lambda(H))) = (-1)^{n-k}[\lambda^{j}](\chi_{H}(-\lambda) - (-\lambda)^n).$$}

We also note that:\\

\textbf{Theorem ~\ref{ShellableThm2}} \emph{Let $H$ be a $k$-uniform hypergraph on $n$ vertices, and let $v$ be a vertex of $H$.  Suppose that the edge set of $H$ consists of all possible hyperedges of size $k$ containing the vertex $v$.  Then $\Lambda(H)$ is shellable.}\\

\textbf{Corollary ~\ref{Homology_Complete_Hypergraph2}} \emph{Let $H$ be a $k$-uniform hypergraph on $n$ vertices, and let $v$ be a vertex of $H$.  Suppose that $H$ consists of all possible hyperedges of size $k$ containing the vertex $v$.  Then the homology of $\Lambda(H)$ is nonzero only in dimension $n-k-1$, and the dimension of $H_{n-k-1}(\Lambda(H))$ equals the sum of the absolute values of the coefficients of $\chi_{H}(\lambda)$ minus one.  Moreover, the dimension of the $j^{th}$ Hodge piece in the Hodge decomposition of $\Lambda(H)$ is 
$$\dim(H_{n-k-1}^{(j)}(\Lambda(H))) = (-1)^{n-k}[\lambda^{j}](\chi_{H}(-\lambda) - (-\lambda)^n).$$}

As a corollary to Theorem ~\ref{ShellableThm2}, we are able to obtain a basis for $\tilde{H}^{n-k-1}(\Lambda(H), \Z)$, and this result is stated in Corollary ~\ref{Cohomology_Basis2}.

\hspace{0.5in}In Section 4, we begin our study of the cyclic coloring complex of the complete $k$-uniform hypergraph with the result:\\

\textbf{Theorem ~\ref{Vertex1Hypergraph}}  \emph{Let $H$ be the $k$-uniform hypergraph on $n$ vertices with edge set consisting of all possible hyperedges of size $k$ containing the vertex 1.  Then the dimension of $HC_r(\Delta(H))$ is nonzero for $n-k-1 \geq r \geq -1$ and is given by
$$\dim(HC_{r}(\Delta(H))) = \binom{n-1}{r+1}.$$}

\hspace{0.5in}We can define an action of $S_{r+2}$ on $\Delta_{r}$.  Namely, if $\sigma \in S_{r+2}$, then $\sigma \cdot (B_1, \hdots, B_{r+2}) = (B_{\sigma^{-1}(1)}, \hdots, B_{\sigma^{-1}(r+2)})$, and this action then makes $C_r$ into an $S_{r+2}$-module, where $C_r$ is the vector space over a field of characteristic zero with basis $\Delta_r$.  Let $\Delta(E_n)$ denote the cyclic coloring complex of the complete graph with looped edges.  Using a result from Crown~\cite{cr}, we will obtain the following result:\\

\textbf{Theorem ~\ref{Sn_module}} \emph{The $S_n$-module structure of $HC_r(\Delta(E_n))$ is $S^{\lambda}$ where $\lambda = (n-r-1, 1^{r+1})$.  Moreover, this is the $S_n$-module structure of the multilinear part of $HC_r(\C[x_1, \hdots, x_n])$.}\\

This theorem will give us the following corollary:

\textbf{Corollary ~\ref{Sn_module_corollary}}  \emph{Let $H$ be a hypergraph on $n$ vertices, let $v$ be a vertex of $H$, and let the hyperedges of $H$ be all possible subsets of $[n]$ of size $k$ that contain vertex $v$.  Then the $S_n$-module structure of $HC_{r}(\Delta(H))$ is $S^{\lambda}$ where $\lambda = (n-r-1, 1^{r+1})$.}\\

\hspace{0.5in}In order to obtain the result of Theorem 4.6, we will define a complex $\Delta(H)^C$, whose $r$-faces consist of all ordered set partitions $[B_1, \hdots , B_{r+2}]$, where none of the $B_i$ contain a hyperedge of $H$ and where $1 \in B_1$.  We compute the dimensions of the homology groups of this complex, $HC_{r}(\Delta(H)^C)$, for $r \geq n-k$.  \\

\textbf{Theorem ~\ref{Homology_Complement}}  \emph{Let $H$ be the complete $k$-uniform hypergraph on $n$ vertices.  For $n-2 \geq r > n-k$,
$$\dim(HC_{r}(\Delta(H)^C)) = \binom{n-1}{r+1} $$
and
$$\dim(HC_{n-k}(\Delta(H)^C)) = \binom{n-1}{n-k-1} + \binom{n-1}{n-k+1}.$$}

This gives, as a result, the dimensions of the multilinear parts of the cyclic homology groups of $\C[x_1, \hdots ,x_n]/\{x_{i_1} \hdots x_{i_k} \mid i_1 \hdots i_k$ is a hyperedge of $H \}$. \\ 

\textbf{Corollary ~\ref{cyclic_homology}}  \emph{For the complete $k$-uniform hypergraph on $n$ vertices, $H$, the dimension of the multilinear part of the $r^{th}$ cyclic homology group of $\C[x_1, \hdots ,x_n]/\{ x_{i_1} \hdots x_{i_k} \mid i_1\hdots i_k $ is a hyperedge of $H \}$ is $\binom{n-1}{r+1}$ for $n-k \leq r \leq n-2$ and $\binom{n-1}{n-k-1} + \binom{n-1}{n-k+1}$ for $r = n-k$.}\\

\textbf{Theorem ~\ref{MainResult}}  \emph{Let $H$ be a complete $k$-uniform hypergraph.  Then
$$\dim(HC_{n-k-1}(\Delta(H))) = \binom{n}{n-k}.$$}

Further, for $k = n-1$ and $k=n-2$, we determine the dimensions of $HC_{r}(\Delta(H))$ for all $r$:\\

\textbf{Theorem ~\ref{special_case_1}}  \emph{Let $H$ be the complete $(n-1)$-uniform hypergraph on $n$ vertices.  Then
$$\dim(HC_{0}(\Delta(H))) = \binom{n}{1} = n$$
and
$$\dim(HC_{-1}(\Delta(H))) = \binom{n}{0} = 1.$$}

and\\

\textbf{Theorem ~\ref{special_case_2}}  \emph{Let $H$ be the complete $(n-2)$-uniform hypergraph on $n$ vertices.  Then for $ -1 \leq r \leq 1$, 
$$\dim(HC_{r}(\Delta(H))) = \binom{n}{r+1}.$$}

\section{Preliminaries}

\begin{definition}
A \emph{hypergraph}, $H$, is an ordered pair, $(V,E)$, where $V$ is a set of vertices and $E$ is a set of subsets of $V$.  A \emph{hyperedge} of $H$ is an element of $E$.  A hypergraph is said to be \emph{uniform of rank $k$}, or \emph{$k$-uniform}, if all of its hyperedges have size $k$.  A $k$-uniform hypergraph is \emph{complete} if every subset of size $k$ of $V$ is a hyperedge of $H$.
\end{definition}

\hspace{0.5in}Throughout this paper, $H$ will denote a hypergraph whose vertex set $V$ is $\{1,\hdots,n\}$.  We define the coloring complex of a hypergraph, $H$, following the presentation in Jonsson~\cite{jo}.

\hspace{.5in}Let $ (B_{1}, \hdots, B_{r+2}) $ be an ordered partition
of $ \{ 1,\hdots, n \} $ where at least one of the $B_{i}$ contains a
hyperedge of $H$.  Further, let $\Lambda_{r}(H)$ be the set of ordered partitions
of length $r+2$ in which at least one part contains a hyperedge of $H$, and let  $V_{r}$ be the vector space over a field of characteristic zero with basis $\Lambda_{r}$(H). 

\begin{definition} The \emph{coloring complex} of $H$, denoted $\Lambda(H)$, has as its $r$-faces the elements of the set $\Lambda_r(H)$ and boundary map $\delta_r: V_r \rightarrow V_{r-1}$ given by



\begin{center}
$\delta_r( (B_{1}, \hdots, B_{r+2}) ) := \displaystyle
\sum_{i=1}^{r+1} (-1)^i ( B_1, \hdots, B_{i} \cup B_{i+1}, \hdots,
B_{r+2} ) $.
\end{center}

\end{definition}

Notice that $\delta_{r-1} \circ \delta_{r} = 0$.  Then:

\begin{definition}
The $r^{th}$ homology group of $\Lambda(H)$ is $H_{r}(\Lambda(H)) =
\ker(\delta_{r})/\im(\delta_{r+1})$.
\end{definition}

It is worth noting that Hultman~\cite{hu} defined a complex that includes both Steingr\'{\i}msson's coloring complex and the coloring complex of a hypergraph as a special case.

\hspace{0.5in}In Section 3, we will show that for a complete $k$-uniform hypergraph $H$, $\Lambda(H)$ is shellable, and hence, Cohen-Macaulay.  We will thus need the following definitions:  

\begin{definition}
A simplicial complex $\Delta$ is \emph{Cohen-Macaulay} over a ring $R$ if $\tilde{H}_{i}(\link_{\Delta}(\sigma);R) = 0$ for all $\sigma \in \Delta$ and $i < \dim(\link_{\Delta}(\sigma))$.
\end{definition}

\begin{definition}
A simplicial complex is \emph{pure} if all of its maximal faces have the same dimension.
\end{definition}

Let $F$ and $G$ be sets and suppose $F \subseteq G$.  The \emph{Boolean interval}, denoted $[F, G]$, is the set $\{ H \mid F \subseteq H \subseteq G \}$.  We let $\bar{F}$ denote the Boolean interval $[\emptyset, F]$. 

\begin{definition}
A \emph{shellable} complex, $\Lambda$, is a complex whose facets can be ordered $F_1, \hdots, F_m$ so that the subcomplex $( \bigcup_{i=1}^{l-1} \bar{F_i} ) \cap \bar{F_l}$ is pure and has dimension $(\dim F_l) -1$ for all $l = 2, \hdots, m$.  If $\Lambda$ is shellable, the ordering of the facets $F_1, \hdots, F_m$ is called a \emph{shelling}.
\end{definition}

\hspace{0.5in}In our proof that $\Lambda(H)$ is shellable for a complete $k$-uniform hypergraph ($k>n/2$), we will use Proposition 2.5 from Bj\"orner and Wachs~\cite{bw} which we include here, along with a necessary definition.

\begin{definition}
Let $F_1, \hdots, F_m$ be a shelling of $\Lambda$.  The \emph{restriction} of facet $F_l$, denoted $\mathcal{R}(F_l)$, is the set $\{ x \in F_l \mid F_l - \{x\} \in \bigcup_{i=1}^{l-1} \bar{F_i}\}$.
\end{definition}

\textbf{Proposition 2.5} [Bj\"orner, Wachs~\cite{bw}]  \emph{Given an ordering $F_1, \hdots, F_m$ of the facets of $\Lambda$ and a map $\mathcal{R}: \{F_1, \hdots, F_m\} \rightarrow \Lambda$, the following are equivalent:
\begin{enumerate}
\item $F_1, \hdots, F_m$ is a shelling and $\mathcal{R}$ its restriction map,
\item $\left \{ 
\begin{array}{cc} 
\Lambda = \bigsqcup_{i=1}^{m}[\mathcal{R}(F_i), F_i], \hspace{0.05in} and \\
\mathcal{R}(F_i) \subseteq F_j \hspace{0.05in} implies \hspace{0.05in} i \leq j, \hspace{0.05in} for \hspace{0.05in} all \hspace{0.05in} i, j.
\end{array}
\right.$
\end{enumerate}
}

\hspace{.5in}In Section 3, we will also determine a basis for the unique nontrivial cohomology group of $\Lambda(H)$.  The cohomology groups of $\Lambda(H)$ are defined in the usual manner.  Please see Munkres~\cite{mk} for more information on the computation of cohomology groups.

\hspace{.5in}In Section 4, we will study the cyclic coloring complex of a complete $k$-uniform hypergraph.  To define the cyclic coloring complex, we first must define an equivalence relation on the elements of $\pm \Lambda_r(H)$:

\hspace{.5in}Let $\sigma \in S_{r+2}$ be the $(r+2)$-cycle $(1, 2, \hdots, r+2)$.  Define $\Delta_r(H) = \pm \Lambda_r(H) / \sim$, where $\sim$ is defined by $(B_1, \hdots, B_{r+2}) \sim (-1)^{r+1}(B_{\sigma(1)}, \hdots, B_{\sigma(r+2)})$.  Let $[B_1, \hdots, B_{r+2}]$ denote the equivalence class containing $(B_1, \hdots, B_{r+2})$.  We will represent each equivalence class of $\Delta_r(H)$ by the unique representative that has $1 \in B_1$.  Let
$$\partial_r([B_1, \hdots, B_{r+2}]) := \sum_{i=1}^{r+1}(-1)^{i+1}[B_1, \hdots, B_i \cup B_{i+1}, \hdots, B_{r+2}] + (-1)^{r+3}[B_1 \cup B_{r+2}, B_2, \hdots, B_{r+1}].$$
It is straightforward to check that $\partial$ is well-defined on equivalence classes.

\begin{definition}
The \emph{cyclic coloring complex} of $H$, $\Delta(H)$, is the sequence
\begin{center}
$ \cdots \rightarrow C_{r} \stackrel{\partial_{r}}{\rightarrow} C_{r-1}
\stackrel{\partial_{r-1}}{\rightarrow} ...
\stackrel{\partial_1}{\rightarrow} C_{0}
\stackrel{\partial_0}{\rightarrow} C_{-1}
\stackrel{\partial_{-1}}{\rightarrow} 0$
\end{center}

where $C_{r}$ is the vector space over a field of characteristic
zero.
\end{definition}

Notice that $\partial_{r-1} \circ \partial_r = 0$, so then:

\begin{definition}
The $r^{th}$ homology group of $\Delta(H)$ is $HC_{r}(\Delta(H)) = \ker(\partial_r)/\im(\partial_{r+1})$.
\end{definition}

\hspace{.5in}As mentioned in Crown~\cite{cr}, the motivation for the definition of the cyclic coloring complex comes from cyclic homology.  See Loday~\cite{l1} for more information on cyclic homology.  

\hspace{.5in}In a couple of our arguments, we will consider the homology of the quotient of two cyclic coloring complexes, so we will define this quotient now:

\hspace{.5in}Consider the cyclic coloring complex of a hypergraph $H$, $\Delta(H)$, and consider a subcomplex, $\Delta(I)$, of $\Delta(H)$, where $I$ is a subhypergraph of $H$.  Then $\Delta_r(H)/\Delta_r(I)$ will consist of the partitions $[B_1, \hdots, B_{r+2}]$ of $\Delta(H)$ where none of the $B_l$ contain a hyperedge of $I$.  Thus, we obtain the sequence of complexes:
$$\Delta(I) \hookrightarrow \Delta(H) \stackrel{j}{\rightarrow} \Delta(H)/\Delta(I)$$
where $i$ is the inclusion map and $j$ is the quotient map.  From the homology of the pair, $(\Delta(H), \Delta(I))$, this then induces the long exact sequence:
$$\cdots \rightarrow HC_{r}(\Delta(I)) \stackrel{i_{*}}{\rightarrow} HC_{r}(\Delta(H)) \stackrel{j_{*}}{\rightarrow} HC_{r}(\Delta(H)/\Delta(I)) \stackrel{\partial_{*}}{\rightarrow} HC_{r-1}(\Delta(I)) \rightarrow \cdots$$
where $i_{*}$ is the map induced by the inclusion $\Delta(I) \hookrightarrow \Delta(H)$, $j_{*}$ is the map induced by the quotient map $j$, and $\partial_{*}$ is the map induced by the boundary map $\partial$.

\hspace{.5in}One of our results will relate the rank of the unique nontrivial homology group of $\Lambda(H)$ to the chromatic polynomial of $H$, $\chi_{H}(\lambda)$.  So we include the definition of $\chi_{H}(\lambda)$ here:

\begin{definition}
A proper $\lambda$-coloring of $H$ is a function $f: V \rightarrow \{ 1, \hdots \lambda \}$ such that for each hyperedge, $e$, of $H$ there exist at least two vertices $v_1$ and $v_2$ in $e$ such that $f(v_1) \neq f(v_2)$.  The \emph{chromatic polynomial} of $H$, denoted $\chi_{H}(\lambda)$, is the polynomial whose value at $\lambda$ gives the number of proper $\lambda$-colorings of $H$.
\end{definition}







\section{The Coloring Complex}

\hspace{.5in}Suppose $k \geq 3$.  In this section, we will show that for $k > n/2$, the coloring complex of a complete $k$-uniform hypergraph is shellable, and we will give a formula for the rank of the unique nontrivial homology group of the coloring complex in terms of the chromatic polynomial of the associated hypergraph.  




\begin{theorem}\label{ShellableThm1}Let $H$ be the complete $k$-uniform hypergraph on $n$ vertices and let $k > n/2$.  Then $\Lambda(H)$ is shellable of dimension $n-k-1$.
\end{theorem}

\begin{proof}

To show that $\Lambda(H)$ is shellable, we will give an ordering of the the facets of $\Lambda(H)$ and will show that there is a unique minimal new face introduced at the $i^{th}$ step.  This will allow us to deduce that $\Lambda(H)$ is a disjoint union of Boolean intervals.  From here it will follow from Proposition 2.5 in Bj\"orner and Wachs~\cite{bw} that $\Lambda(H)$ is shellable.

\hspace{0.5in}Notice that the facets of $\Lambda(H)$ are ordered set partitions, $F=(B_1, \hdots, B_{n-k+1})$ of $[n]$ where exactly one block has size $k$ and the other blocks are singleton blocks.  Let $h(F)$ be the index of the block of size $k$ and let $w(F) = w_1 w_2 \hdots w_{n-k}$ be the permutation of the elements of $[n] \backslash B_{h(F)}$ obtained by listing these elements in the order that they appear in $F$.  Let $G = (B_1', \hdots, B_{n-k+1}')$ be a facet of $\Lambda(H)$, and consider the following ordering of the facets:
\begin{itemize}
\item If $h(F) < h(G)$, then $F$ appears before $G$.
\item If $h(F) = h(G)$ and $B_{h(F)}$ appears before $B'_{h(G)}$ in the lexicographic order on $k$-sets, then $F$ appears before $G$.
\item If $h(F) = h(G)$ and $B_{h(F)} = B'_{h(G)}$, then $F$ appears before $G$ if and only if $w(F)$ appears before $w(G)$ in the lexicographic order on permutations.
\end{itemize}

\hspace{0.5in}We now define a map $\mathcal{R}$ from the set of facets to $\Lambda(H)$.  Let the descent set of $w(F)$ be $\des(w(F)) = \{ d_1, \hdots, d_l \}$.  If $A_F = \{ i \mid i < h(F) \hspace{.1in} \text{and} \hspace{.1in} d_i \in \des(w(F)) \}$ is a nonempty set, then let $t = \max \{ i \mid i \in A_F \}$.  If $w_{h(F)} > \max\{B_{h(F)}\}$, then 
$$\mathcal{R}(F) = (w_1 \cup \hdots \cup w_{d_1}, w_{d_1+1} \cup \hdots \cup w_{d_2}, \hdots,  w_{d_{t}} \cup \hdots \cup w_{h(F)-1}, B_{h(F)} \cup w_{h(F)} \cup \hdots \cup w_{d_{t+1}}, \hdots ,w_{d_l} \cup \hdots \cup w_{n-k}).$$  
Otherwise, 
$$\mathcal{R}(F) = (w_1 \cup \hdots \cup w_{d_1}, w_{d_{1}+1} \cup \hdots \cup w_{d_2}, \hdots,  w_{d_{t}} \cup \hdots \cup w_{h(F)-1}, B_{h(F)}, w_{h(F)} \cup \hdots \cup w_{d_{t+1}}, \hdots ,w_{d_l} \cup \hdots \cup w_{n-k}).$$
On the other hand, suppose $A_F = \emptyset$.  If $w_{h(F)} > \max\{B_{h(F)}\}$, then
$$\mathcal{R}(F) = (w_1 \cup \hdots \cup w_{h(F)-1}, B_{h(F)} \cup w_{h(F)} \cup \hdots \cup w_{d_1}, w_{d_1 +1} \cup \hdots \cup w_{d_2}, \hdots, w_{d_l} \cup \hdots \cup w_{n-k}).$$
Otherwise,
$$\mathcal{R}(F) = (w_1 \cup \hdots \cup w_{h(F)-1}, B_{h(F)}, w_{h(F)} \cup \hdots \cup w_{d_1}, w_{d_1 + 1} \cup \hdots \cup w_{d_2}, \hdots ,w_{d_l} \cup \hdots \cup w_{n-k}).$$

\hspace{0.5in}Let $F_1, \hdots, F_m$ be the ordering of the facets of $\Lambda(H)$ under the above ordering.  We first show that $\mathcal{R}(F_i)$ is a new face introduced at the $i^{th}$ step.  Notice that for a facet $F_i = (w_1, \hdots, w_{h(F_i)-1}, B_{h(F_i)}, w_{h(F_i)}, \hdots, w_{n-k})$ that it follows from the ordering that for $s \neq h(F_i)-1$, the face 
$$(w_1, \hdots, w_s \cup w_{s+1}, \hdots, w_{h(F_i)-1},B_{h(F_i)}, w_{h(F_i)}, \hdots, w_{n-k})$$ 
is in $\bigcup_{j=1}^{i-1} \bar{F_j}$ if and only if $w_s > w_{s+1}$.  Since $w_1 < \hdots < w_{d_{1}}$, it follows that the face 
$$(w_1 \cup \hdots \cup w_{d_1}, w_{d_1 + 1}, \hdots, w_{h(F_i)-1}, B_{h(F_i)}, w_{h(F_i)}, \hdots, w_{n-k})$$
 is not in $\bigcup_{j=1}^{i-1} \bar{F_j}$.  It then follows similarly that for $A_{F_i} \neq \emptyset$
 $$(w_1 \cup \hdots \cup w_{d_1}, w_{d_1+1} \cup \hdots \cup w_{d_2}, \hdots,  w_{t} \cup \hdots \cup w_{h(F_i)-1}, B_{h(F_i)}, w_{h(F_i)} \cup \hdots \cup w_{t+1}, \hdots ,w_{d_l} \cup \hdots \cup w_{n-k})$$ 
is not in $\bigcup_{j=1}^{i-1} \bar{F_j}$.  Further, notice that if $w_{h(F_i)} < \max\{B_{h(F_i)}\}$, then the face 
$$(w_1, \hdots, w_{h(F_i)-1}, B_{h(F_i)} \cup w_{h(F_i)}, w_{h(F_i)+1}, \hdots, w_{n-k})$$ 
is in $\bigcup_{j=1}^{i-1} \bar{F_j}$ since it is also a face of the facet 
$$(w_1, \hdots, w_{h(F_i)-1}, (B_{h(F_i)} - \{ \max \{B_{h(F_i)}\} \}) \cup w_{h(F_i)}, \max\{B_{h(F_i)}\}, w_{h(F_i)+1}, \hdots, w_{n-k}).$$  
So if $w_{h(F_i)} < \max\{B_{h(F_i)}\}$, then 
$$\mathcal{R}(F_i) = (w_1 \cup \hdots \cup w_{d_1}, w_{d_1+1} \cup \hdots \cup w_{d_2}, \hdots,  w_{t} \cup \hdots \cup w_{h(F_i)-1}, B_{h(F_i)}, w_{h(F_i)} \cup w_{t+1}, \hdots ,w_{d_l} \cup \hdots \cup w_{n-k})$$ 
is a new face.  Otherwise, 
$$\mathcal{R}(F_i) = (w_1 \cup \hdots \cup w_{d_1}, w_{d_1+1} \cup \hdots \cup w_{d_2}, \hdots,  w_{t} \cup \hdots \cup w_{h(F_i)-1}, B_{h(F_i)} \cup w_{h(F_i)} \cup \hdots \cup w_{t+1}, \hdots ,w_{d_l} \cup \hdots \cup w_{n-k})$$
is a new face of the complex.  We can similarly argue that if $A_{F_i} = \emptyset$, $\mathcal{R}(F_i)$ is a new face of the complex.

\hspace{0.5in}Notice that the minimality of $\mathcal{R}(F_i)$ follows from the above observations as well as by noting that the face 
$$(w_1, \hdots, w_{h(F_i)-2}, w_{h(F_i)-1} \cup B_{h(F_i)}, w_{h(F_i)}, \hdots w_{n-k})$$
is a face of the facet 
$$(w_1, \hdots, w_{h(F_i)-2}, B_{h(F_i)}, w_{h(F_i)-1}, w_{h(F_i)}, \hdots, w_{n-k}).$$

Since $\mathcal{R}(F_i)$ is the unique minimal new face introduced at the $i$-th step, by induction, we have:
$$\Lambda(H) = \bigsqcup_{i=1}^{m} [\mathcal{R}(F_i), F_i].$$

\hspace{0.5in}From the above argument, it follows that $\mathcal{R}(F_i) \subseteq F_j$ implies $i \leq j$ for all $i, j$.  Then by Proposition 2.5 in Bj\"orner and Wachs~\cite{bw}, the ordering $F_1, \hdots, F_m$ is a shelling order and $\mathcal{R}$ is its restriction map, and $\Lambda(H)$ is shellable. \end{proof}

\hspace{0.5in}As noted in Bj\"orner ~\cite{bj}, shellable complexes are Cohen-Macaulay.  Thus, $\Lambda(H)$ has a unique nontrivial homology group.  By Theorem 4.1 of Long and Rundell~\cite{lr}, we obtain the following corollary:


\begin{corollary}\label{Homology_Complete_Hypergraph}
If $H$ is the complete, $k$-uniform hypergraph on $n$ vertices, then the homology of $\Lambda(H)$ is nonzero only in dimension $n-k-1$, and the dimension of $H_{n-k-1}(\Lambda(H))$ equals the sum of the absolute values of the coefficients of $\chi_{H}(\lambda)$ minus one.  Moreover, the dimension of the $j^{th}$ Hodge piece in the Hodge decomposition of $\Lambda(H)$ is 
$$\dim(H_{n-k-1}^{(j)}(\Lambda(H))) = (-1)^{n-k}[\lambda^{j}](\chi_{H}(-\lambda) - (-\lambda)^n).$$
\end{corollary}

\hspace{0.5in}From Theorem ~\ref{ShellableThm1}, we may obtain a basis for the reduced cohomology group $\tilde{H}^{n-k-1}(\Lambda(H), \Z).$  Notice that $\Lambda(H)$ is a shellable pure complex, and let $\Gamma$ be the set of homology facets of the shelling from the proof of Theorem ~\ref{ShellableThm1}, i.e. the set of facets $F$ of $\Lambda(H)$ such that $\mathcal{R}(F) = F$.  By the definition of $\mathcal{R}$, $\Gamma$ is the set of facets $F$ of $\Lambda(H)$ such that $\max\{B_{h(F)}\} > w_{h(F)}$ and the permutation $w(F) = w_1 \hdots w_{n-k}$ satisfies $w_1 > \hdots > w_{h(F) - 1}$ and $w_{h(F)} > \hdots > w_{n-k}$. 

\hspace{0.5in}Following the notation in Bj\"orner and Wachs~\cite{bw}, for each $F \in \Gamma$, let $\sigma^{F}$ denote a $(n-k-1)$-cochain defined by
$$\sigma^{F}(G) = \begin{cases} 1 & \text{if $G=F$} \\ 0 & \text{if $G \neq F$} \end{cases} $$
for $G \in \Lambda_{n-k-1}(H)$.  The fact that $F$ is a facet of $\Lambda(H)$ implies that $\sigma^F$ is a cocycle and therefore determines a cohomology class $[\sigma^{F}]$ if $\tilde{H}^{n-k-1}(\Lambda(H), \Z)$.  Theorem 4.3 in Bj\"orner and Wachs~\cite{bw} then gives the following corollary to Theorem ~\ref{ShellableThm1}:


\begin{corollary}\label{Cohomology_Basis}Let $\Gamma$ be the set of facets $F$ of $\Lambda(H)$ such that $\max\{B_{h(F)}\} > w_{h(F)}$ and the permutation $w(F) = w_1 \hdots w_{n-k}$ satisfies $w_1 > \hdots > w_{h(F) - 1}$ and $w_{h(F)} > \hdots > w_{n-k}$.  Then the classes $[\sigma^{F}]$, for $F \in \Gamma$, are a basis of $\tilde{H}^{n-k-1}(\Lambda(H), \Z)$.
\end{corollary}

\hspace{0.5in}Notice that the condition that $k > n/2$ is necessary in Theorem ~\ref{ShellableThm1}.  If $n/2 \geq k > 2$, the complete $k$-uniform hypergraph contains a pair of disjoint edges.  By Proposition 7 of Breuer, et al.~\cite{bd}, $\Lambda(H)$ is not Cohen-Macaulay and hence not shellable.

\hspace{0.5in}From the shelling order and the argument given in the proof of Theorem ~\ref{ShellableThm1} we also have the following result and corollaries:

\begin{theorem}\label{ShellableThm2}Let $H$ be a $k$-uniform hypergraph on $n$ vertices, and let $v$ be a vertex of $H$.  Suppose that the edge set of $H$ consists of all possible hyperedges of size $k$ containing the vertex $v$.  Then $\Lambda(H)$ is shellable.
\end{theorem}

\begin{corollary}\label{Homology_Complete_Hypergraph2}
Let $H$ be a $k$-uniform hypergraph on $n$ vertices, and let $v$ be a vertex of $H$.  Suppose that $H$ consists of all possible hyperedges of size $k$ containing the vertex $v$.  Then the homology of $\Lambda(H)$ is nonzero only in dimension $n-k-1$, and the dimension of $H_{n-k-1}(\Lambda(H))$ equals the sum of the absolute values of the coefficients of $\chi_{H}(\lambda)$ minus one.  Moreover, the dimension of the $j^{th}$ Hodge piece in the Hodge decomposition of $\Lambda(H)$ is 
$$\dim(H_{n-k-1}^{(j)}(\Lambda(H))) = (-1)^{n-k}[\lambda^{j}](\chi_{H}(-\lambda) - (-\lambda)^n).$$
\end{corollary}

\begin{corollary}\label{Cohomology_Basis2}Let $H$ be a $k$-uniform hypergraph on $n$ vertices, and let $v$ be a vertex of $H$.  Suppose that the edge set of $H$ consists of all possible hyperedges of size $k$ containing the vertex $v$.  Let $\Gamma$ be the set of facets $F$ of $\Lambda(H)$ where $v \in B_{h(F)}$, $\max\{B_{h(F)}\} > w_{h(F)}$, and the permutation $w(F)$ satisfies $w_1 > \hdots > w_{h(F)-1}$ and $w_{h(F)} > \hdots > w_{n-k}$.  Then the classes $[\sigma^{F}]$, for $F \in \Gamma$, are a basis of $\tilde{H}^{n-k-1}(\Lambda(H), \Z)$.
\end{corollary}


\section{The Cyclic Coloring Complex}

\hspace{.5in} In this section, we discuss the homology of the cyclic coloring complex of a complete $k$-uniform hypergraph.  In the proof of Theorem ~\ref{Vertex1Hypergraph}, we will use a spectral sequence argument similar to the proof of Theorem 3.2 in Crown~\cite{cr}.  Thus, before presenting Theorem ~\ref{Vertex1Hypergraph} and its proof, we provide a summary of some of the results of Crown~\cite{cr} that we will need in the proof of Theorem ~\ref{Vertex1Hypergraph}.

\textbf{Theorem 3.2}  [Crown~\cite{cr}]  \emph{Let $\Delta(E_n)$ be the cyclic coloring complex of the complete graph with looped edges at each vertex.  The dimension of the $r$-th homology group of $\Delta(E_n)$, $HC_{r}(\Delta(E_n))$, is $\binom{n-1}{r+1}$.}

\hspace{.5in}The proof of Theorem 3.2 begins by considering the function
$$F([B_1, \hdots, B_{r+2}]) = \mid B_1 \mid$$
where $B_1$ is the block containing the vertex 1.  Let $\Delta_{r}^{m}$ denote the elements of $\Delta_r$ such that
$$F([B_1, \hdots, B_{r+2}]) = m,$$
and let $\Delta(E_{n}^m)$ denote the complex formed by the chains in $\Delta_{r}^m$, $-1 \leq r \leq n-2$.  Notice that $f$ gives a grading of each $\Delta_r(E_n)$.  Define $\Delta_{r}^{(m)}(E_n) = \bigcup_{m \leq i \leq n} \Delta_{r}(E_{n}^i)$.  When the boundary map $\partial$ is applied to an element of $\Delta_{r}(E_{n}^i)$ the result is a signed sum of elements where $\mid B_1 \mid = i$ or $\mid B_1 \mid = i+1$.  Therefore, $\partial(\Delta_{r}^{(m)}(E_n)) \subseteq \Delta_{r-1}^{m}(E_n)$, and thus $\partial$ respects the grading.

\hspace{.5in}The proof of Theorem 3.2 then uses a spectral sequence argument to determine the dimensions of the homology groups of $\Delta(E_n)$ and follows the construction and notation of Chow~\cite{ch}.  Let $C_{r,m}$ be the vector space with basis $\Delta_{r}^{(m)}(E_n)$.  Then $E_{r,m}^0 = C_{r,m}/C_{r,m+1}$, and in particular, $E_{r,m}^0$ is the vector space with basis $\Delta_{r}(E_{n}^m)$ and $C_r \cong \bigoplus_{m=1}^n E_{r,m}^0$.  Further, $\partial$ induces a map
$$\partial^{0}: \bigoplus_{m=1}^n E_{r,m}^0 \rightarrow E_{r-1,m}^0$$
where $\partial^{0}(E_{r,m}^0) \subseteq E_{r-1,m}^0$ for all values of $r,m$.  One can then define:
$$E_{r,m}^1 = HC_{r}(E_{r,m}^0) = \frac{\ker \partial^{0}: E_{r,m}^0 \rightarrow E_{r-1,m}^0}{\text{im } \partial^{0}: E_{r+1,m}^0 \rightarrow E_{r,m}^0}.$$
Further, $\partial$ induces a map:
$$\partial^{1}: E_{r,m}^1 \rightarrow E_{r-1,m+1}^1$$
and we can define
$$E_{r,m}^2 = HC_{r}(E_{r,m}^1) = \frac{\ker \partial^{1}: E_{r,m}^1 \rightarrow E_{r-1,m+1}^1}{\text{im } \partial^{1}: E_{r+1,m-1}^1 \rightarrow E_{r,m}^1}.$$
In the proof of Theorem 3.2, $E_{r,m}^2 = 0$ for all values of $r$ and $m$ which implies that $HC_{r}(\Delta(E_n)) = \bigoplus_{-1 \leq r \leq n-2} E_{r,m}^1$.

\hspace{.5in}To compute the dimension of $HC_r(\Delta(H))$, Crown notes that the elements of $\Delta^{m}(H)$ can be partitioned into subcomplexes determined by the elements of $B_1$.  Each of these subcomplexes has the homology of a Boolean algebra of $n-m$ elements.  A homology representative of the unique nontrivial homology group of the subcomplex is then $\sum_{\sigma \in S_{n-m}} \sgn(\sigma)[B_1, a_{\sigma(1)}, \hdots, a_{\sigma(n-m)}]$, where $\{a_1, \hdots, a_{n-m}\}$ are the elements of the set $[n] \backslash B_1$.  There are $\binom{n-1}{m-1}$ subcomplexes of $\Delta^m(E_n)$, and since $E_{r,m}=0$ for all values of $r$ and $m$, it follows that the dimension of $HC_{r}(\Delta(E_n))$ is given by $\binom{n-1}{r+1}$.

\begin{theorem}\label{Vertex1Hypergraph}Let $H$ be the $k$-uniform hypergraph on $n$ vertices with edge set consisting of all possible hyperedges of size $k$ containing the vertex 1.  Then the dimension of $HC_r(\Delta(H))$ is nonzero for $n-k-1 \geq r \geq -1$ and is given by
$$\dim(HC_{r}(\Delta(H))) = \binom{n-1}{r+1}.$$
\end{theorem}

\begin{proof}  Consider the function 
$$f([B_1, \hdots, B_{r+2}]) = \mid B_1 \mid$$
where $B_1$ is the block containing 1.  Let $\Delta_{r}^{m}(H)$ denote the elements in $\Delta_r(H)$ where 
$$f([B_1, \hdots, B_{r+2}]) = m,$$
and let $\Delta^{m}(H)$ denote the complex formed by the elements in $\Delta_{i}^{m}(H)$, $-1 \leq i \leq n-2$.  Notice that $f$ gives a grading of each $\Delta_{r}(H)$ and that the boundary map $\partial$ respects the grading.  We will use a spectral sequence argument to determine the dimensions of the homology groups of $\Delta(H)$.  

\hspace{0.5in}Notice that $\Delta^{m}(E_n) = \Delta^{m}(H)$ for $n-k+1 \leq m \leq n$.  As in the proof of Theorem 3.2 in Crown~\cite{cr}, we can see that the elements of $\Delta^{m}(H)$ can be partitioned into subcomplexes determined by the elements of $B_1$, and each subcomplex has the homology of a Boolean algebra of $n-m$ elements.  In particular, each subcomplex has a unique nontrivial homology group of rank one with homology representative given by $\displaystyle \sum_{\sigma \in S_{n-m}} \sgn(\sigma) [ B_1, a_{\sigma(1)}, \hdots, a_{\sigma(n-m)}]$, where  $\{ a_1, \hdots, a_{n-m} \}$ are the elements of the set $[n] \backslash B_1$.

\hspace{0.5in}Consider $\partial^{1}(\sum_{\sigma \in S_{n-m}} \sgn(\sigma) [ B_1, a_{\sigma(1)}, \hdots, a_{\sigma(n-m)}])$.  Recall that $\partial^{1}$ is the map obtained by taking those terms of the boundary map, $\partial$, in which the size of $B_1$ increased by one.  So,
$$\partial^{1}([B_1, a_1, \hdots, a_{n-m}]) = [B_1 \cup a_1, a_2, \hdots, a_{n-m}] + (-1)^{n-m}[B_1 \cup a_{n-m}, a_1, \hdots, a_{n-m-1}].$$

\hspace{.5in}Let $\pi$ be a permutation in $S_{n-m}$ and let $\pi(1)=i$.  There exists a unique permutation $\tau$ in $S_{n-m}$ for which
$$\tau(n-m)=i \hspace{0.1in} \text{and} \hspace{0.1in} \tau(1) = \pi(2), \hdots, \tau(n-m-1)=\pi(n-m).$$

Notice that when we apply $\partial^{1}$ to $\sgn(\pi)[B_1, a_{\pi(1)}, \hdots, a_{\pi(n-m)}]$ and $\sgn(\tau)[B_1, a_{\tau(1)}, \hdots, a_{\tau(n-m)}]$, each resulting sum will have the term $[B_1 \cup a_i, a_{\pi(2)}, \hdots, a_{\pi(n-m)}]$.  It remains to show that the coefficients of these terms cancel.  The coefficient of $[B_1 \cup a_i, a_{\pi(2)}, \hdots, a_{\pi(n-m)}]$ in $\partial^{1}(\sgn(\pi)[B_1, a_{\pi(1)}, \hdots, a_{\pi(n-m)}])$ is $\sgn(\pi)$.  Note that $\sgn(\tau) = (-1)^{n-m-1}\sgn(\pi)$, and thus the coefficient of $[B_1 \cup a_i, a_{\pi(2)}, \hdots, a_{\pi(n-m)}]$ in $\partial^{1}(\sgn(\tau)[B_1, a_{\tau(1)}, \hdots, a_{\tau(n-m)}])$ is $(-1)^{n-m-1}\sgn(\pi)*(-1)^{n-m} = -\sgn(\pi)$.  Thus, 
$$\partial^{1}(\sum_{\sigma \in S_{n-m}} \sgn(\sigma) [ B_1, a_{\sigma(1)}, \hdots, a_{\sigma(n-m)}]) = 0.$$
The same argument holds for $\partial^2, \partial^3, \hdots$, and thus the spectral sequence collapses.

\hspace{0.5in}It then follows that to determine the dimension of the homology group $HC_{r}(\Delta(H))$, we must relate $r$ to $m$ and determine the number of subcomplexes of $\Delta^{m}(H)$.  Notice that $r = (n-2)-(m-1)$ and that the number of subcomplexes of $\Delta^{m}(H)$ is given by the number of ways of forming a subset of size $m-1$ from a set of size $n-1$.  Thus,
$$\dim(HC_{r}(\Delta(H)) = \binom{n-1}{r+1}. \qedhere$$
\end{proof}

\hspace{.5in}Notice that the proof of Theorem ~\ref{Vertex1Hypergraph} shows that the homology representatives of the $r^{th}$ homology group of $\Delta(H)$ are indexed by the subsets of size $n-r-2$ of $\{ 2, \hdots, n \}$.  Namely, for each subset, $A$, of size $n-r-2$ of $\{2, \hdots,n\}$, we obtain one homology representative of $HC_{r}(\Delta(H))$, $\displaystyle \sum_{\sigma \in S_{r+1}} \sgn(\sigma) [ A \cup \{1\}, a_{\sigma(1)}, \hdots, a_{\sigma(r+1)}]$, where  $\{ a_1, \hdots, a_{r+1} \}$ is the complement of $A$ in $\{2, \hdots, n \}$.

\hspace{0.5in}We can define an action of $S_{r+2}$ on $\Delta_{r}$, and this action then makes $C_r$ into an $S_{r+2}$-module.  Namely, if $\sigma \in S_{r+2}$, then $\sigma \cdot (B_1, \hdots, B_{r+2}) = (B_{\sigma^{-1}(1)}, \hdots, B_{\sigma^{-1}(r+2)})$.  It is possible to describe the $S_n$-module structure of the $r$-th homology group of $\Delta(E_n)$, and hence, the $S_n$-module structure of the $r$-th homology group of $\Delta(H)$, where $H$ is a complete $k$-uniform hypergraph.


\begin{theorem}\label{Sn_module}
The $S_n$-module structure of $HC_r(\Delta(E_n))$ is $S^{\lambda}$ where $\lambda = (n-r-1, 1^{r+1})$.  Moreover, this is the $S_n$-module structure of the multilinear part of $HC_r(\C[x_1, \hdots, x_n])$.
\end{theorem}

\begin{proof}  For each subset $A$, of size $n-r-2$ of $\{2, \hdots, n\}$, we obtain one homology representative of $HC_{r}(\Delta(E_n))$, namely $\sum_{\sigma \in S_{r+1}} \sgn(\sigma)[A \cup \{1\}, a_{\sigma(1)}, \hdots, a_{\sigma_({r+1})}]$.  Let $B_1 = A \cup \{1\}$.  Since $1 \in B_1$, consider first the $S_{n-1}$-module structure of $HC_{r}(\Delta(E_n))$, where $S_{n-1}$ is acting on $\{2, \hdots, n\}$.  The homology representative $\sum_{\sigma \in S_{r+1}} \sgn(\sigma)[B_1, a_{\sigma(1)}, \hdots, a_{\sigma_({r+1})}]$, is invariant under permutations of the elements of the set $B_1 \backslash \{1\}$, and if the elements of the set $ \{a_1, \hdots, a_{r+1} \}$ are permuted by $\tau$, then the homology representative is mapped to 
$$\sgn(\tau)\sum_{\sigma \in S_{r+1}} \sgn(\sigma)[B_1, a_{\sigma(1)}, \hdots, a_{\sigma_({r+1})}].$$  Thus,
$$HC_{r}(\Delta(E_n))\downarrow_{S_{n-1}} = (S^{(n-r-2)} \otimes S^{(1^{r+1})}).$$
By the Littlewood-Richardson Rule, we then have:
$$HC_{r}(\Delta(E_n))\downarrow_{S_{n-1}} = S^{(n-r-1, 1^r)} \oplus S^{(n-r-2, 1^{r+1})}.$$

\hspace{0.5in}If $\lambda$ is a partition of $n$, then $S^{\lambda}\downarrow_{S_{n-1}} \cong \bigoplus_{\lambda^{-}} S^{\lambda^{-}}$ where $\lambda^{-}$ is a partition obtained from $\lambda$ by removing one block.  We similarly use the notation $\lambda^{+}$ for a partition obtained from $\lambda$ by adding one block.

\hspace{0.5in}We wish to show that $V = S^{(n-r-1, 1^{r+1})}$ is the only representation such that $V \downarrow_{S_{n-1}} = S^{(n-r-1, 1^{r})} \oplus S^{(n-r-2, 1^{r+1})}$, for $n \geq 5$.  (We treat the cases $n=3$ and $n=4$ separately below.)  Suppose $V = \sum_{\mu} c_{\mu} S^{\mu}$.  Notice that for $\lambda_1 = (n-r-1, 1^r)$, the possibilities for $\lambda_1^+$ are $(n-r, 1^r)$, $(n-r-1, 1^{r+1})$, and $(n-r-1, 2, 1^{r-1})$.  For $\lambda_2 = (n-r-2, 1^{r+1})$, the possibilities for $\lambda_2^{+}$ are $(n-r-1, 1^{r+1})$, $(n-r-2, 1^{r+2})$, and $(n-r-2, 2, 1^{r})$.  So,
$$V = c_1S^{(n-r, 1^r)} + c_2S^{(n-r-1, 1^{r+1})} + c_3S^{(n-r-1, 2, 1^{r-1})} + c_4 S^{(n-r-2, 1^{r+2})} + c_5 S^{(n-r-2, 2, 1^r)}.$$

\hspace{0.5in}Consider the partitions $(n-r-1, 2, 1^{r-1})$ and $(n-r-2, 2, 1^r)$.  In the first case, notice that one of the partitions that is obtained by removing a block is the partition $(n-r-1,2,1^{r-2})$.  When we restrict $V$, the only other possible partition that could contribute a term to cancel the term corresponding to $(n-r-1,2,1^{r-2})$ is the partition $(n-r-2,2,1^r)$.  However, if we remove a block from the partition $(n-r-2,2,1^r)$, we obtain the partition $(n-r-3,2,1^r)$, $(n-r-2,1^{r+1})$, or $(n-r-2,2,1^{r-1})$.  Since none of these partitions are equal to $(n-r-1,2,1^{r-2})$, $c_3 = 0$.  By a similar argument, $c_5 = 0$. 

\hspace{0.5in}Now consider the partitions $(n-r, 1^r)$ and $(n-r-2, 1^{r+2})$.  For the former, if we remove a block from the partition, we obtain either the partition $(n-r, 1^{r-1})$ or the partition $(n-r-1,1^r)$.  When we restrict $V$, the only other possible partition that could contribute terms to cancel the terms corresponding to $(n-r,1^{r-1})$ and $(n-r-1,1^r)$ is the partition $(n-r-2,1^{r+2})$.  However, notice that when we remove a block from the partition $(n-r-2,1^{r+2})$ we obtain either the partition $(n-r-3,1^{r+1})$ or the partition $(n-r-2,1^r)$.  Since neither of these partitions is equal to $(n-r-1,1^r)$ or $(n-r,1^{r-1})$, $c_1 = c_4 = 0$.  Thus, for $n \geq 5$, the restriction $S^{(n-r-1, 1^{r+1})}\downarrow_{S_{n-1}}$ equals $S^{\lambda_1} \oplus S^{\lambda_2}$.

\hspace{0.5in}Suppose that $n=3$.  For $r = 1$ and $r=-1$, the above argument holds.  In particular, the argument shows that the $S_n$-module structure of $HC_{1}(\Delta(E_3))$ is $S^{(1^3)}$, and the $S_n$-module structure of $HC_{-1}(\Delta(E_3))$ is $S^{(3)}$.  Suppose then that $r=0$.  The elements $[12, 3]$ and $[13,2]$ are a set of homology representatives for $HC_{0}(\Delta(E_3))$.  It is straightforward to see that the character on the conjugacy class indexed by the identity element equals 2, the character on the conjugacy class indexed by cycle type $(2,1)$ is 0, and the character on the conjugacy class indexed by cycle type $(3)$ is $-1$.   Thus, the $S_n$-module structure of $HC_{0}(\Delta(E_n))$ is $S^{(2,1)}$.

\hspace{0.5in}Suppose $n=4$.  The cases $r=2$ and $r=-1$ can be proven using the restriction argument above.  For the case $r=0$, $[123, 4]$, $[124, 3]$, and $[134, 2]$ are a set of homology representatives for $HC_{0}(\Delta(E_4))$.  It is straightforward to see that the character on the conjugacy class indexed by the identity element is 3, the character on the conjugacy class indexed by cycle type $(2,1,1)$ is 1, the character on the conjugacy class indexed by cycle type $(2,2)$ is $-1$, the character on the conjugacy class indexed by cycle type $(3,1)$ is zero, and the character on the conjugacy class indexed by cycle type $(4)$ equals $-1$.  Therefore, the $S_4$-module structure of $HC_{0}(\Delta(E_4))$ is $S^{(3,1)}$.  For the case $r=1$, the argument is similar, and it can be seen that the $S_4$-module structure of $HC_{1}(\Delta(E_4))$ is $S^{(2,1,1)}$.  \end{proof}


\begin{corollary}\label{Sn_module_corollary}
Let $H$ be a hypergraph on $n$ vertices, let $v$ be a vertex of $H$, and let the hyperedges of $H$ be all possible subsets of $[n]$ of size $k$ that contain vertex $v$.  Then the $S_n$-module structure of $HC_{r}(\Delta(H))$ is $S^{\lambda}$ where $\lambda = (n-r-1, 1^{r+1})$.
\end{corollary}

\hspace{.5in}Let $H$ be the complete $k$-uniform hypergraph on $n$ vertices.  Notice that $\Delta(H)$ is a subcomplex of $\Delta(E_n)$.  Let $\Delta(H)^C = \Delta(E_n)/\Delta(H)$.  Also, notice that if we apply the boundary map of $\Delta(H)^C$ to a partition $[B_1, \hdots,B_{r+2}]$ and if one of the terms in the image contains an edge of $H$, then this particular term is equal to zero.  To compute the homology of $\Delta(H)$, we will first compute the homology of $\Delta(H)^C$.


\begin{theorem} \label{Homology_Complement} Let $H$ be the complete $k$-uniform hypergraph on $n$ vertices.  

For $n-2 \geq r > n-k$,
$$\dim(HC_{r}(\Delta(H)^C)) = \binom{n-1}{r+1} $$
and
$$\dim(HC_{n-k}(\Delta(H)^C)) = \binom{n-1}{n-k-1} + \binom{n-1}{n-k+1}.$$
\end{theorem}

\hspace{.5in}Before presenting the proof of Theorem ~\ref{Homology_Complement}, we will need the following definitions and lemma from Crown~\cite{cr}.  

\hspace{.5in}Let $T_n$ be a tree on $n$ vertices.  Let the root of the tree be labeled 1, and label the other vertices $2, \hdots, n$ so that each parent node has a smaller vertex label than each of its children.  Consider listing each edge of $T_n$ by placing the smaller vertex first and order the edges in lexicographic order.  Let $\Delta(T_n^{(0,l)})$ be the complex formed by the elements $[B_1, \hdots, B_{r+2}]$ where none of the $B_i$ contain one of the first $l$ edges.

\textbf{Lemma 3.3} [Crown~\cite{cr}]  Given a tree on $n$ vertices, for $r \leq n-3$, 
$$\dim(\Delta(T_n^{(0,l)})) = \binom{n-(l+1)}{(r+2)-(l+1)}.$$

We now prove Theorem ~\ref{Homology_Complement}:

\begin{proof}

For the first part of the theorem, consider the following long exact sequence:

\begin{center}
\begin{tabular}{ccccccccc}
$0$ & $\rightarrow$ & $HC_{n-2}(\Delta(H))$ &
$\rightarrow$ & $HC_{n-2}(\Delta(E_{n}))$ &
$\rightarrow$ &
$HC_{n-2}(\Delta(H)^C)$ & $\rightarrow$ & \\
& & $HC_{n-3}(\Delta(H))$ & $\rightarrow$ &
$HC_{n-3}(\Delta(E_n)$ &
$\rightarrow$ & $HC_{n-3}(\Delta(H)^C)$ & $\rightarrow$ &  \\
&  & $HC_{n-4}(\Delta(H))$ & $\rightarrow$ &
$HC_{n-4}(\Delta(E_n))$ & $\rightarrow$ &
$HC_{n-4}(\Delta(H)^C)$ & $\rightarrow$ &
\end{tabular}
\newline \centerline{$\vdots$} \newline
\end{center}

Notice that since $HC_{r}(\Delta(H)) = 0$ for $r > n-k-1$, by exactness, $\dim HC_{r}(\Delta(H)^C) = \dim HC_{r}(\Delta(E_n)) = \binom{n-1}{r+1}$ for $n-2 \geq r > n-k$.

\hspace{.5in}We will now prove the second statement of the theorem.  Let $W_1, \hdots, W_{\binom{n}{k}}$ be the subsets of $\{1, \hdots, n \}$ of size $k$, listed in lexicographic order.  Let $\Delta(E_{n}^{(0, W_l)})$ be the complex formed be the chains $[B_1, \hdots, B_{r+2}]$ where for all $i$, $1 \leq i \leq l$, the elements of $W_i$ are not in the same block of the partition.  Let $\Delta(E_{n}^{(1, W_l)})$ be the complex formed by the chains $[B_1, \hdots, B_{r+2}]$ where for all $i$, $1\leq i \leq l-1$, the elements of $W_i$ are not in the same block of the partition, but the elements of $W_l$ are in the same $B_j$, for some $j$, $1 \leq j \leq r+2$.  Notice then that:
$$\Delta(E_n^{(0,W_{l-1})})/\Delta(E_n^{(1,W_l)}) = \Delta(E_n^{(0,W_l)}) .$$
Using this notation, $\Delta(H)^C = \Delta(E_n^{(0,W_{\binom{n}{k}})})$.  We will compute the homology of $\Delta(H)^C$ by sequentially computing the homology of $\Delta(E_n^{(0,W_l)})$.

\hspace{.5in}  Let $H'$ be the $k$-uniform hypergraph on $n$ vertices with edge set consisting of all possible hyperedges of size $k$ containing the vertex 1.  Let $p = \binom{n-1}{k-1}$.  Since $1 \in W_l$ for $1 \leq l \leq p$, $\Delta(H')^C = \Delta(E_n^{(0,W_{p})})$.  We begin by using the following long exact sequence to calculate the homology of $\Delta(H')^C$:

\begin{center}
\begin{tabular}{ccccccccc}
$0$ & $\rightarrow$ & $HC_{n-2}(\Delta(H'))$ &
$\stackrel{\alpha_{n-2}}{\rightarrow}$ & $HC_{n-2}(\Delta(E_{n}))$ &
$\rightarrow$ &
$HC_{n-2}(\Delta(H')^C)$ & $\rightarrow$ & \\
& & $HC_{n-3}(\Delta(H'))$ & $\stackrel{\alpha_{n-3}}{\rightarrow}$ &
$HC_{n-3}(\Delta(E_n)$ &
$\rightarrow$ & $HC_{n-3}(\Delta(H')^C)$ & $\rightarrow$ &  \\
&  & $HC_{n-4}(\Delta(H'))$ & $\stackrel{\alpha_{n-4}}{\rightarrow}$ &
$HC_{n-4}(\Delta(E_n))$ & $\rightarrow$ &
$HC_{n-4}(\Delta(H')^C)$ & $\rightarrow$ &
\end{tabular}
\newline \centerline{$\vdots$} \newline
\end{center}

By the same argument as in the first part of this proof, we can see that the dimension of $HC_{r}(\Delta(H'))^C$ for $n-2 \geq r \geq n-k+1$ is $\binom{n-1}{r+1}$.  So consider $r = n-k$.  We noted after the proof of Theorem~\ref{Vertex1Hypergraph} that the homology representatives of the $r^{th}$ homology group of $\Delta(H')$ are indexed by the subsets of size $n-r-2$ of $\{ 2, \hdots, n \}$.  As noted in the proof of Theorem 3.2 in Crown~\cite{cr}, these are the same as the homology representatives of the $r^{th}$ homology group of $\Delta(E_n)$.  Thus $\alpha_r$ is bijective for all $r \leq n-k-1$.   By the above argument, it suffices to consider the exact sequence:

\begin{center}
\begin{tabular}{ccccccc}
$0$ & $\rightarrow$ & $HC_{n-k}(\Delta(E_n))$ & $\rightarrow$ & 
$HC_{n-k}(\Delta(H')^C)$ & $\stackrel{\phi_{n-k}}{\rightarrow}$ & \\ 
$HC_{n-k-1}(\Delta(H'))$ & $\stackrel{\alpha_{n-k-1}}{\rightarrow}$ & $HC_{n-k-1}(\Delta(E_n))$ & $\rightarrow$ & 0 & & \\
\end{tabular}
\end{center}

Since $\alpha_{n-k-1}$ is injective, the image of $\phi_{n-k}$ is zero.  Therefore, the dimension of $HC_{n-k}(\Delta(H')^C)$ equals the dimension of the kernel of $\phi_{n-k}$.  By exactness, the dimension of the kernel of $\phi_{n-k}$ equals the dimension of $HC_{n-k}(\Delta(E_n))$.  So the dimension of $HC_{n-k}(\Delta(H')^C)$ is $\binom{n-1}{n-k+1}$.



\hspace{.5in}Now we will sequentially compute $HC_{r}(\Delta(H)^C)$.  Let $l > p$.  Notice that $HC_{r}(\Delta(E_n^{(1,W_l)})$ is zero for $r \geq n-k$.  So by exactness, for $r > n-k$,
$$\dim(HC_{r}(\Delta(E_n^{(0,W_{l})})) = \dim(HC_{r}(\Delta(H')^C)) = \binom{n-1}{r+1}.$$

Further, we know $HC_{r}(\Delta(H')^C) = 0$ for $r \leq n-k-1$.  It suffices then to consider the the exact sequence:

\begin{center}
\begin{tabular}{ccccccc}
0 & $\rightarrow$ & $HC_{n-k}(\Delta(H')^C)$ & $\rightarrow$ & $HC_{n-k}(\Delta(E_n^{(0,W_{p+1})}))$ & $\rightarrow$ \\
 $HC_{n-k-1}(\Delta(E_n^{(1,W_{p+1})}))$ &
$\stackrel{\alpha_{n-k-1}}{\rightarrow}$ & 0 &
$\rightarrow$ &
$HC_{n-k-1}(\Delta(E_n^{(0,W_{p+1})}))$ & $\rightarrow$ & \\
$HC_{n-k-2}(\Delta(E_n^{(1,W_{p+1})}))$ &
$\stackrel{\alpha_{n-k-2}}{\rightarrow}$ & 0 &
$\rightarrow$ &
$HC_{n-k-2}(\Delta(E_n^{(0,W_{p+1})}))$ & $\rightarrow$ & \\
\end{tabular}
\newline \centerline{$\vdots$} \newline
\end{center}

By exactness, $HC_{r}(\Delta(E_n^{(0,W_{p+1})})) \cong HC_{r-1}(\Delta(E_n^{(1,W_{p+1})}))$ for $ n-k-1 \geq r \geq 0$.  Consider the complex $\Delta(E_n^{(1,W_{p+1})})$.  Notice that the homology of this complex is equal to the homology of the complex $\Delta(T_{n-k+1}^{(0,1))})$, where $T_{n-k+1}$ is the tree on $n-k+1$ vertices, with root labeled 1, and edges $(1,2), (1,3), \hdots, (1, n-k+1)$, and where vertices $1$ and $2$ are not in the same block of a chain $[B_1, \hdots, B_{r+2}]$.  Let $\{ a_1, \hdots, a_{n-k-1} \}$ be the elements of the complement of $W_{p+1}$ in $\{ 2, \hdots, n \}$ listed in increasing order.  The isomorphism is given by mapping 1 to 1, $W_{p+1}$ to 2, $a_1$ to 3, $\hdots$, $a_{n-k-1}$ to $n-k+1$.  By Lemma 3.3 of Crown~\cite{cr}, the dimension of $HC_{n-k-1}(\Delta(E_n^{(1,W_{p+1})}))$ is equal to $\binom{(n-k+1)-(1+1)}{((n-k-1)+2)-(1+1)} = 1$, and therefore, $\dim(HC_{n-k}(\Delta(E_n^{(0,W_{p+1})}))) = 1 + \binom{n-1}{n-k+1}$.
 
 \hspace{.5in}We claim that for each subsequent set $W_l$ removed, the contribution to $\dim HC_{n-k}(\Delta(H)^C)$ will be one.  Notice that for all $l$, $ p + 1 \leq l \leq \binom{n}{k}$, the homology of the complex $\Delta(E_n^{(1, W_l)})$ is equal to the homology of the complex $\Delta(T_{n-k+1}^{(0,i)})$ for some $i$.  By Lemma 3.3 of Crown, the dimension of $HC_{n-k-1}(\Delta(E_n^{(0,W_l)}))$ is equal to $\binom{(n-k+1)-(l+1)}{((n-k-1)+2)-(l+1)} = 1$.  By exactness, for each subsequent set $W_l$ removed, the contribution to $\dim(HC_{n-k}(\Delta(H)^C))$ will be one.  Since there are $\binom{n-1}{k}$ sets $W_l$ of size $k$ of $[n-1]$, $\dim(HC_{n-k}(\Delta(H)^C)) = \binom{n-1}{k} + \binom{n-1}{n-k+1} = \binom{n-1}{n-k-1} + \binom{n-1}{n-k+1}$. \qedhere
\end{proof}

\hspace{.5in}Notice that the elements of $\Delta_r(H)^C$ are in bijection with the cyclic words $[D_1, \hdots,D_{r+2}]$ where $D_i \in \C[x_1, \hdots,x_n]/\{ x_{i_1} \hdots x_{i_k} \mid i_1 \hdots i_k $ is a hyperedge of $H \}$ and $[D_1, \hdots,D_{r+2}]$ is an ordered partition of $x_1\hdots x_n$ with $x_1 \in D_1$.  It then follows that we have the following corollary:


\begin{corollary} \label{cyclic_homology} For the complete $k$-uniform hypergraph on $n$ vertices, $H$, the dimension of the multilinear part of the $r^{th}$ cyclic homology group of $\C[x_1, \hdots,x_n]/\{ x_{i_1} \hdots x_{i_k} \mid i_1 \hdots i_k $ is a hyperedge of $H \}$ is $\binom{n-1}{r+1}$ for $n-k \leq r \leq n-2$ and $\binom{n-1}{n-k-1} + \binom{n-1}{n-k+1}$ for $r = n-k$.
\end{corollary}

We will now determine the dimension the $(n-k-1)^{st}$ homology group of $\Delta(H)$ for a complete $k$-uniform hypergraph:


\begin{theorem}\label{MainResult}
Let $H$ be a complete $k$-uniform hypergraph.  Then
$$\dim(HC_{n-k-1}(\Delta(H))) = \binom{n}{n-k}.$$
\end{theorem}

\begin{proof}

Consider the following long exact sequence:

\begin{center}
\begin{tabular}{cccccccc}
0 & $\rightarrow$ & $HC_{n-k}(\Delta(E_n))$ & $\rightarrow$ & $HC_{n-k}(\Delta(H)^C)$ & $\stackrel{\phi_{n-k}}{\rightarrow}$ & \\
$HC_{n-k-1}(\Delta(H))$ & $\stackrel{\alpha_{n-k-1}}{\rightarrow}$ & $HC_{n-k-1}(\Delta(E_n))$ & $\stackrel{\beta_{n-k-1}}{\rightarrow}$ & $HC_{n-k-1}(\Delta(H)^C)$ & $\rightarrow$ & $\cdots$ \\
\end{tabular}
\end{center}

As noted in the proof of Theorem ~\ref{Homology_Complement}, each of the homology representatives of $HC_{r}(\Delta(E_n))$ corresponds to a subset, $A$, of $\{2, \hdots, n \}$ of size $n-r-2$.  Since the set $\{1\} \cup A$ is an edge of $H$, each of the homology representatives of $HC_{n-k-1}(\Delta(E_n))$ is mapped to zero by the map $\beta_{n-k-1}$.  Therefore, the dimension of the kernel of $\beta_{n-k-1}$ is $\binom{n-1}{n-k}$.  By exactness and Theorem 3.2 of Crown~\cite{cr}, the dimension of the kernel of $\phi_{n-k} = \binom{n-1}{n-k+1}$.  Thus,
\begin{align*}
\dim(HC_{n-k-1}(\Delta(H))) & =  \left( \dim(HC_{n-k}(\Delta(H)^C)) - \binom{n-1}{n-k+1} \right) + \binom{n-1}{n-k} \\
 & =  \binom{n-1}{n-k-1} + \binom{n-1}{n-k} \\
 & =  \binom{n}{n-k} \hspace{3.9in} \qedhere
 \end{align*} 
\end{proof}

When $k = n-1$ and $k=n-2$, we have the results:

\begin{theorem} \label{special_case_1}Let $H$ be the complete $(n-1)$-uniform hypergraph on $n$ vertices.  Then
$$\dim(HC_{0}(\Delta(H))) = \binom{n}{1} = n$$
and
$$\dim(HC_{-1}(\Delta(H))) = \binom{n}{0} = 1.$$
\end{theorem}

\begin{proof}

Note that there $\binom{n}{1}$ elements in $\Delta_0(H)$ and each of these elements is mapped to zero under $\partial_0$.  Thus, $\dim(HC_0(\Delta(H))) = n$.  Since $\Delta_{-1}(H) = \{ [12 \hdots n ] \}$, it follows that $\dim(HC_{-1}(\Delta(H))) = 1$. \end{proof}



\begin{theorem} \label{special_case_2}Let $H$ be the complete $(n-2)$-uniform hypergraph on $n$ vertices.  Then for $ -1 \leq r \leq 1$, 
$$\dim(HC_{r}(\Delta(H))) = \binom{n}{r+1}.$$
\end{theorem}

\begin{proof}

From Theorem ~\ref{MainResult}, we know that $\dim(HC_{1}(\Delta(H))) = \binom{n}{2}$.   The set $\Delta_{1}(H)$ consists of all ordered partitions $ [B_1, B_2, B_3] $ where $\{1\} \in B_1$ and where one of the $B_i$ is a hyperedge of $H$.  It follows then that there are $2 \binom{n-1}{n-3} + 2 \binom{n-1}{n-2} = 2 \binom{n}{n-2}$ elements in $\Delta_{1}(H)$.  Thus the dimension of the image of $\partial_1$ is $\binom{n}{2}$.  The dimension of the kernel of $\partial_0$ equals the cardinality of $\Delta_0(H)$.  Since $\Delta_{0}(H)$ consists of all ordered partitions $[B_1, B_2]$ where $\{1\} \in B_1$ and where one of the $B_i$ contains a hyperedge of $H$, there are $ n + \binom{n}{2}$ elements in $\Delta_0(H)$.  So, the dimension of $HC_{0}(\Delta(H)) = \binom{n}{1}$.  It is clear that the dimension of $HC_{-1}(\Delta(H)) = \binom{n}{0}$. \end{proof}



\textbf{Acknowledgments}

\hspace{0.5in} The author would like to thank the anonymous referees for their thoughtful feedback and suggestions.  In particular, the author would like to thank one of the anonymous referees for the suggestion of the shelling order in the proof of Theorem ~\ref{ShellableThm1}.  The author would also like to thank Jane H. Long for several helpful conversations on the importance of the condition $k < n/2$ in the proof of Theorem ~\ref{ShellableThm1}.  

\bibliography{biblio}

\end{document}